\documentclass{amsart}
\usepackage{amsmath, amsthm, amssymb}
\usepackage{verbatim}

\renewcommand{\d}{\partial}
\newcommand{\dbar}{\overline{\partial}}
\newcommand{\ddbar}{\sqrt{-1}\d\overline{\d}}
\newcommand{\nddbar}{\frac{\sqrt{-1}}{2\pi}\d\overline{\d}}

\newtheorem{thm}{Theorem}
\newtheorem{prop}[thm]{Proposition}
\newtheorem{lem}[thm]{Lemma}
\newtheorem{cor}[thm]{Corollary}

\theoremstyle{definition}
\newtheorem{defn}[thm]{Definition}
\newtheorem{rem}[thm]{Remark}

\renewcommand{\[}{\begin{equation}}
\renewcommand{\]}{\end{equation}}

\newcommand{\be}{\beta}

\newcommand{\la}{\lambda}

\newcommand{\La}{\Lambda}

\newcommand{\vep}{\varepsilon}

\newcommand{\NN}{\mathbb{N}}

\newcommand{\QQ}{\mathbb{Q}}
\newcommand{\RR}{\mathbb{R}}
\newcommand{\CC}{\mathbb{C}}

\newcommand{\PP}{\mathbb{P}}
\newcommand{\sH}{\mathcal{H}}

\newcommand{\sM}{\mathcal{M}}
\newcommand{\fg}{\mathfrak{g}}

\newcommand{\Nu}{\mathcal{V}}
\newcommand{\sW}{\mathcal{W}}
\newcommand{\sG}{\mathcal{G}}
\newcommand{\sL}{\mathcal{L}}
\newcommand{\sA}{\mathcal{A}}
\newcommand{\DF}{\mathrm{DF}}
\newcommand{\sJ}{\mathcal{J}}

\newcommand{\dvols}{\frac{\omega_s^n}{n!}}

\newcommand{\dvj}{\frac{\omega_i^n}{n!}}

\newcommand{\dvo}{\frac{\omega_0^n}{n!}}
\newcommand{\dvos}{\frac{\omega_{0,s}^n}{n!}}
\newcommand{\dvok}{\frac{(k\omega_0)^n}{n!}}

\newcommand{\dvvi}{\frac{\omega_i^n}{V_in!}}

\newcommand{\dvjs}{\frac{\omega_{i,s}^n}{n!}}
\newcommand{\dvvjs}{\frac{\omega_{i,s}^n}{V_in!}}

\newcommand{\cFut}{\mathrm{Fut}_c}
\newcommand{\Fut}{\mathrm{Fut}}
\newcommand{\ctFut}{\mathrm{Fut}_{c,(1-t)\psi}}

\title{On coupled constant scalar curvature K\"ahler metrics}
\author[V. V. Datar]{Ved V. Datar}
\author[V. P. Pingali]{Vamsi Pritham Pingali}
\address{Department of Mathematics, Indian Institute of Science, Bangalore, India - 560012}
\email{vvdatar@iisc.ac.in, vamsipingali@iisc.ac.in}

\begin{document}

\maketitle
\begin{abstract} 
We provide a moment map interpretation for the coupled K\"ahler-Einstein equations introduced in \cite{Jakobintro}, and in the process introduce a more general system of equations, which we call coupled cscK equations. A differentio-geometric formulation of the corresponding Futaki invariant is obtained and a notion of K-polystability is defined for this new system. Finally, motivated by a result of Sz\'ekelyhidi, we prove that if there is a solution to our equations, then small $K$-polystable perturbations of the underlying complex structure and polarizations also admit coupled cscK metrics. 
\end{abstract}

\section{Introduction}

\indent Our aim in this paper is to study a set of metrics satisfying some coupled equations on a K\"ahler manifold, that generalise constant scalar curvature K\"ahler (cscK) metrics, and the coupled K\"ahler-Einstein metrics studied in \cite{Jakobintro,Jakobtorpap, Jakobhom, Futakicoupled, cKE, Taka}. Throughout the paper we fix a polarized tuple $(M,(L_i))$, i.e., \ an $n$-dimensional K\"ahler manifold $M$ with ample line bundles $L_0,L_1,\cdots,L_m$, and we denote the line bundle  $\displaystyle \otimes_{i=0}^mL_i$ by $L$. We are interested in metrics $\omega_i\in  2\pi c_1(L_i)$ for $i=0,\cdots,m$ that satisfy
\begin{align}\label{coupled csck}
\frac{\omega_0^n}{V_0} &= \cdots = \frac{\omega_m^n}{V_m}\\
S_{\omega_0} &= \mathrm{tr}_{\omega_0}\omega + \hat S,\nonumber
\end{align}
where $S_{\omega_0}$ is the scalar curvature of $\omega_0$, $V_i = (2\pi L_i)^n/n!$, $\omega = \omega_0+\omega_1+\cdots+\omega_m$ and $\hat S$ is a computable constant, namely, $$\hat S = n\frac{(-K_M - L)\cdot L_0^{n-1}}{L_0^n}.$$ In particular, if $M$ is Fano and $L = -K_M$, then $\hat S = 0$, and it is easy to show that the above system reduces to the coupled K\"ahler Einstein system, viz., 
\begin{gather}
\mathrm{Ric}(\omega_0)=\mathrm{Ric}(\omega_1)=\ldots =\mathrm{Ric}(\omega_m) = \omega.
\label{coupledKEequation}
\end{gather}
 In analogy to the relationship between the K\"ahler-Einstein problem and the cscK problem,  we refer to  $(\omega_0,\omega_1,\cdots,\omega_m)$ solving \eqref{coupled csck} as  {\em coupled cscK metrics}, and we say that $(M,(L_i))$ admits coupled cscK metrics. The reader should however be forewarned that coupled cscK metrics will in general not have constant scalar curvatures (unless of course $m=0$).

Our main result is that coupled cscK metrics have a moment map interpretation. To describe the setting, we fix a K\"ahler form $\omega_0 \in 2\pi c_1(L_0)$ and Hermitian metrics $h_1,\cdots,h_m$ on the underlying smooth bundles $L_1,\cdots,L_m$ respectively, and consider a  subspace $\sM \subset \sJ\times\sA_1\times\cdots\sA_m$ of ``integrable tuples" (cf.\ section 3 for details), where $\sJ$ is the space of almost complex structures on $M$ compatible with $\omega_0$ and taming it, and $\sA_i$ is the space of unitary connections on $L_i$. There is a natural almost complex structure $\mathbf{I}$ and a compatible symplectic form $\Omega$ on $\sM$ giving it a formal K\"ahler structure. The theorem alluded to above is the following. 

\begin{thm}\label{thm:momentmapthm}
There exists a group $\sG$ with a Hamiltonian action on $(\sM,\Omega)$ such that if $\mu:\sM\rightarrow Lie(\sG)^*$ is the moment map for the action, then $\mu(J,A_1,\cdots,A_m) = 0$ if and only if  $\omega_i = \sqrt{-1}F_{A_i}$ is a $(1,1)$ form for $i=1,\cdots,m$, and $(\omega_0,\omega_1,\cdots,\omega_m)$ are coupled cscK metrics.

\end{thm}
In the standard moment map picture of Fujiki \cite{Fujiki} and Donaldson \cite{Donaldmom} for the cscK problem, the group of Hamiltonian symplectomorphisms play the role of the gauge group. Inspired by \cite{Mariothesis}, we define our gauge group $\sG$ as a subgroup of the group of unitary automorphisms of the vector bundle $(E=\oplus_{i=1}^m L_i, \oplus_i h_i)$ covering Hamiltonian symplectomorphisms of $(M,\omega_0)$. Details are presented in Section \ref{sec:moment}.

\begin{rem}\label{whythisequation}
Naively, one might want to define a coupled cscK system by simply tracing each of the equations in \eqref{coupledKEequation}. However, unlike \eqref{coupled csck}, to our knowledge such a system does not appear to have a natural moment map interpretation.
\end{rem}

Analogous to the beautiful perturbation results of Br\"onnle \cite{Bron} and Sz\'ekelyhidi \cite{Gabor}, using some techniques in \cite{Mariotipler}, we apply the moment map picture to obtain a deformation result for coupled cscK metrics.
 
\begin{thm}\label{thm:perturb}
Suppose $(M,J,(L_i))$ admits coupled cscK metrics. Then a sufficiently small deformation $(M,J',(L_i'))$ of the complex structure of $(M,J,(L_i))$  admits coupled cscK metrics if it is $K$-polystable. 
\end{thm}
We expect that the converse, namely that existence of coupled cscK metrics implies $K$-polystability should also be true, but we do not get into these considerations in this paper. 
We now provide a brief outline of our paper. In section \ref{sec:moment} we provide a moment map interpretation of the system of equations (\ref{coupled csck}). In section \ref{sec:Futaki} we use our moment map interpretation to give a definition of a coupled Futaki invariant on a normal variety, which vanishes precisely when the tuple admits a coupled cscK metric. We then define the corresponding notion of $K$-polystability, and show that when $L=-K_M$ our definition coincides with the algebro-geometric one in \cite{Jakobintro}. As an aside, we also define a twisted coupled Futaki invariant.  Section \ref{sec:perturbsection} contains the proof of Theorem \ref{thm:perturb}, following closely the proofs in \cite{Gabor, Bron, Mariotipler}.\\

\emph{Acknowledgements} : The authors would like to thank Ruadhai Dervan for some clarifications on his work on twisted cscK metrics, and many useful comments on the first draft of the paper. The second author (Pingali) is  partially supported by an SERB grant : ECR/2016/001356. He is also grateful to the Infosys foundation for the Infosys Young Investigator Award. This work is also partially supported by grant F.510/25/CAS-II/2018(SAP-I) from UGC (Govt. of India).
\section{A moment map interpretation for the coupled cscK equation}\label{sec:moment}
\indent In this section we prove theorem \ref{thm:momentmapthm}. As in the introduction, for $i=0,\cdots,m$, let $h_i$ be metrics on $L_i$ and $\mathcal{A}_{i}$ be the space of $h_i$-unitary connections on $L_i$. Let $-\sqrt{-1}\omega_0$ be the curvature of a fixed connection $A_0 \in\mathcal{A}_0$ on $L_0$ such that $\omega_0$ defines a K\"ahler form with respect to the given complex structure on $M$. Akin to \cite{Mariothesis}, let $\mathcal{G}_i$ be the gauge group of unitary gauge transformations of $(L_i,h_i)$ covering the identity and $\tilde{\mathcal{G}}$ be the group of gauge transformations of $(E=\oplus_{i=1}^mL_i,\oplus_{i=1}^m h_i)$ of the form $g_1\oplus g_2\ldots$ covering Hamiltonian symplectomorphisms of $(M,\omega_0)$. If $\mathcal{H}$ is the group of Hamiltonian symplectomorphisms of $(M,\omega_0)$, then there is a short exact sequence 
$$0\rightarrow \mathcal{G}_1\times \mathcal{G}_2\ldots \rightarrow \tilde{\mathcal{G}} \rightarrow \mathcal{H} \rightarrow 0.$$
Indeed, the last map is onto because of the existence of a horizontal lift of a Hamiltonian vector field.\\
\indent Next, let $\mathcal{N}=\mathcal{J}\times \mathcal{A}_1 \times \mathcal{A}_2\ldots \mathcal{A}_m$ where $\mathcal{J}$ is the space of almost complex structures compatible with and taming $\omega_0$. For ease of notation, we denote an element $(J,A_1,\cdots,A_m)$ simply as the pair $(J,A)$. Note that $A$ can be thought of as a unitary connection on $(E,\oplus_{i=1}^m h_i)$. Let $\mathcal{M}\subset \mathcal{N}$ be the subset consisting of pairs $(J,A)$, such that $J$ is integrable and $\sqrt{-1}F_{A_i}$ is a positive $(1,1)$ form for each $i$. (More accurately, we only deal with the open set consisting of the smooth part of the subset $\mathcal{M}$.) Note that the tangent space $\mathcal{T}_{A_i}\mathcal{A}_i$ is given by $\Lambda^1(M,i\RR)$, which we identify with $\Lambda^1(M,\RR)$. On the other hand, the tangent space of $\mathcal{J}$ is given by $$\mathcal{T}_J\mathcal{J} = \{S\in \mathrm{End}(TM)~|~ SJ+JS = 0,~\omega_0(SX,JY) +\omega_0(JX,SY) = 0 \},$$ and so the tangent space $\mathcal{T}_{(J,A)}\mathcal{N}$ to $\mathcal{N}$ at a point $(J,A)$ is given by pairs $(S,a)$, where  $S\in \mathcal{T}_J\mathcal{J}$ and $a = (a_1,\cdots,a_m)$ with $a_i\in \Lambda^1(M,\RR)$. The tangent space of $\mathcal{M}$ at an integrable point $(J,A)$ is a subspace $T_{(J,A)}\mathcal{M}$ of $T_{(J,A)}\mathcal{N}$ consisting of infinitesimally integrable pairs. There is a natural almost complex structure $\mathbf{I}$ induced on $\mathcal{T}_{(A, J)}\mathcal{N}$, namely $$\mathbf{I}_{(A, J)}(S,a) = (JS,J^*a),$$ where $J^*$ is the dual action $J^*a(v) = a(Jv).$ This complex structure is integrable \cite{Mariothesis}. \\
\indent Taking cue from \cite{fine, vbma} we define a 2-form on $\mathcal{N}$ as follows.
$$\Omega_{(J,A)}((S,a),(T,b)) = -\sum_{i=1}^m \frac{V_0}{V_i}\int_M a_i \wedge b_i\wedge \frac{\omega_{A_i}^{n-1}}{(n-1)!}+\int_M tr(JST) \dvo,$$
where $\omega_{A_i} = \sqrt{-1}F_{A_i}$, and $V_i = (2\pi)^n\frac{L_i^n}{n!} \ \forall \ i\geq 0$. 
\begin{lem}
$\Omega$ is a symplectic form on $\sM$ compatible with $\mathbf{I}$.
\end{lem}
\begin{proof}Firstly note that for each $i$, $\omega_{A_i}(\cdot,J\cdot)$ defines a Riemannian metric. 
Non-degeneracy then follows from the following observation.
\begin{align*}
\Omega_{(J,A)}((S,a),\mathbf{I}(S,a)) &= -\displaystyle\sum_{i=1}^m\frac{V_0}{V_i}\int_M a_i\wedge J^*a_i \wedge \frac{\omega_{A_i}^{n-1}}{(n-1)!} + \int_M\mathrm{Tr}([JS]^2)\dvo\\
&=\sum_{i=1}^m \frac{V_0}{V_i}\int_M|a_i|^2_{g_j}\frac{\omega_{A_i}^n}{n!} + \int_M\mathrm{Tr}(S^2)\dvo,\\
&>0,
\end{align*}
unless $(S,a)$ is the zero tangent vector. To show that $\Omega$ is closed, we first observe  that $\Omega$ is of the form $\displaystyle \pi_{0}^{*} \tilde{\Omega}+ \sum_{i=1}^{m} \pi_{i}^{*} \Omega_{i} $ where $$\Omega_i (a,b) = \displaystyle \int_M \frac{V_0}{V_i}a\wedge b \wedge \frac{\omega_{A_i}^{n-1}}{(n-1)!},~\tilde{\Omega}(S,T)=\int_M tr(JST) \dvo,$$ and for $i=0,1,\cdots,m$, $\pi_i$ is the projection from $\sM$ to the $i^{th}$ factor. These forms are individually closed \cite{vbma, fine, Donaldmom, Fujiki}. The compatibility of $\Omega$  with $\mathbf{I}$ follows from the equation $JS = -SJ$ and the fact that $\omega_{A_i}$ is a $(1,1)$-form with respect to $J$ for every $i$.  
\end{proof}
\indent The group $\tilde{\mathcal{G}}$ acts on $\mathcal{N}$ in a natural manner. Namely, if $\tilde g\in \tilde{\mathcal{G}}$ covers $f\in \mathcal{H}$, then $$\tilde g\cdot (J,A) = (f_*Jf_*^{-1},\tilde g (f^{-1})^*A \tilde g^{-1} - (f^{-1})^*(d\tilde g)\tilde g^{-1}).$$
The following lemma is then obtained by simply tracing through the definitions. 
\begin{lem}
The action of $\tilde{\mathcal{G}}$ restricts to a symplectic action on $(\sM,\Omega)$. 
\end{lem}

Our aim is to show that this action is in fact Hamiltonian and to identify the moment map. To do this, we need to understand the infinitesimal action of $\mathrm{Lie}(\tilde{\sG})$ on the pair $(J,A)$.  Let $\xi \in Lie(\tilde{\mathcal{G}})$ generate the vector field $\tilde \zeta$  on $E$, covering a vector field $\zeta$ on $M$ which is Hamiltonian with respect to $\omega_0$. We let $\zeta^{hor}$ denote the horizontal lift  of $\zeta$ to a vector field on $E$, defined with respect to the connection $A$. If $t^i$ is the local complex coordinate on $L_i$, then it is not difficult to see that 
\begin{equation}\label{hor}
\zeta^{hor} = \zeta + \sum_{i=1}^m\sqrt{-1} i_\zeta A_i~\mathrm{Im}\Big(t^i\frac{\partial}{\partial t^i}\Big),\end{equation}  and hence there exist $H_{\zeta,A_i}\in C^{\infty}(M,\RR)$ such that 
\begin{equation}\label{dec}
\tilde \zeta = \zeta^{hor} - \sum_{i=1}^mH_{\zeta,A_i}~\mathrm{Im}\Big(t^i\frac{\partial}{\partial t^i}\Big).
\end{equation}
We write $H_{\zeta,A} = \mathrm{diag}(H_{\zeta,A_1},\cdots,H_{\zeta,A_m})$ and $\omega_A = \sqrt{-1}F_A = \mathrm{diag}(\omega_{A_1},\cdots,\omega_{A_m})$. We also let $H_{\zeta,0}$ be the Hamiltonian of $\zeta$ with respect to $\omega_0$. Our convention is that $$dH_{\zeta,0} = -i_{\zeta}\omega_0.$$
\begin{lem}\label{inf action}The infinitesimal action $P:\mathrm{Lie}(\tilde \sG)\rightarrow \mathcal{T}_{(J,A)}\sM$ of $\mathrm{Lie}(\tilde G)$ is given by $$P(\xi) := \xi \cdot (J,A) =(\sL_\zeta J,dH_{\zeta,A} + i_\zeta\omega_{A}),$$ where recall once again that we are identifying $T_{A_i}\sA_i$ by $\Omega^1(M,\RR)$, and hence the first term on the right is indeed a real form. In particular, if $\xi$ is in the stabilizer of $A$, then for each $i$, $\zeta$ is Hamiltonian with respect to $\omega_{A_i}$ with Hamiltonian $H_i$.
\end{lem}

\begin{proof}
The second part, namely $\xi\cdot J = L_\zeta J$ is obvious, and so we focus on the connection part. Note that locally if we write $\tilde g \cdot (p,\vec{v}) = (f(p),[g_p]\cdot \vec{v})$ for some diagonal matrix $g_p$ (depending of course on the point $p$), then the action is given by $$g\cdot A_i = [g_p] (f^{-1})^{*}A_i [g_p]^{-1} -(f^{-1})^{*}d[g_p] [g_p]^{-1}.$$ Now suppose $\tilde g_t = e^{t\xi}$ is a path in $\tilde{\sG}$ such that the corresponding vertical part is given locally by $g_{p,t} = e^{\sqrt{-1}t\eta}$, then 
\begin{align*}
\frac{d}{dt}\Big|_{t=0}\tilde g_t\cdot A &= -\mathcal{L}_\zeta A - \sqrt{-1}d\eta\\
&=-di_\zeta A - i_\zeta F_A-\sqrt{-1}d\eta.
\end{align*} Since $\eta$ is the vertical component of $\tilde\zeta$, from \eqref{hor} and \eqref{dec} it is follows that $$\eta = \sqrt{-1}i_\zeta A - H_{\zeta,A},$$ and so $$\sqrt{-1}\xi\cdot A := \frac{d}{dt}\Big|_{t=0}\tilde g_t\cdot A =-i_\zeta F_A + \sqrt{-1}dH_{\zeta,A} = \sqrt{-1}(dH_{\zeta,A} + i_\zeta\omega_A).$$ \end{proof}

We can now state the main result of this section. 

\begin{thm} The action of $\tilde{\mathcal{G}}$ on $\mathcal{M}$ is Hamiltonian with moment map
\begin{gather}
\mu_{(J,A)}(\xi) = \displaystyle \int_M \mathrm{tr}\left(H_{\zeta,A} \Big(\lambda \frac{\omega_A^n}{n!}- \dvo\Big)\right) + \int_W H_{\zeta,0} \Big(S_{\omega_0,J}\dvo - \mathrm{tr}(\omega_A)\wedge \frac{\omega_0^{n-1}}{(n-1)!}\Big),
\label{momentmapdef}
\end{gather}
where $\lambda=\mathrm{diag}(\frac{V_0}{V_1},\frac{V_0}{V_2},\ldots)$ and $\xi$ generates a vector field $\tilde\zeta$ on $E$ covering a vector field $\zeta$ on $M$.
\label{momentmapthm}
\end{thm}

\begin{proof}
We need to show that if $(J(t),A(t))$ is a path in $\sM$ with $A(t) =A +  \sqrt{-1}bt$ and $J'(0) = T$, then 
\begin{align*}\frac{d}{dt}\Big|_{t=0}\mu_{(J(t),A(t))}(\tilde \zeta) &= -\Omega((L_\zeta J,\xi\cdot A), (T,b))\\
&= \int_M \mathrm{tr}\Big(\la\Big(\xi\cdot A\wedge b\wedge\frac{\omega_A^{n-1}}{(n-1)!}\Big)\Big) - \int_M\mathrm{tr}(JL_\zeta J T)\dvo.
\end{align*}
Here $b$ is a diagonal matrix $b = \mathrm{diag}(b_1,\cdots,b_m)$ of real one forms. In \cite{Donaldmom, Fujiki} it is shown that 
\begin{equation}\label{don}
\frac{d}{dt}\Big|_{t=0}\int_M H_{\zeta,0} S_{\omega_0,J(t)}\dvo = -\int_M\mathrm{tr}(JL_\zeta J T)\dvo.
\end{equation}
Next, differentiating the first term in $\mu_{(J(t),A(t))}(\tilde\zeta)$, since $\frac{\omega_{A(t)}}{dt} = -db$, we have 
\begin{align*}
\frac{d}{dt}\Big|_{t=0} \int_M \mathrm{tr}\Big(\la H_{\zeta,A(t)} \frac{\omega_{A(t)}^n}{n!}\Big) &= \int_M\mathrm{tr}\Big(\la\frac{d H_{\zeta,A(t)}}{dt}\Big|_{t=0}\frac{\omega_A^n}{n!}\Big) - \int_M\mathrm{tr}\Big(\la H_{\zeta,A}db\wedge\frac{\omega_A^{n-1}}{(n-1)!}\Big)\\
&=\int_M\mathrm{tr}\Big(\la\frac{d H_{\zeta,A(t)}}{dt}\Big|_{t=0}\frac{\omega_A^n}{n!}\Big) + \int_M \mathrm{tr}\Big(\la\Big(dH_{\zeta,A}\wedge b\wedge\frac{\omega_A^{n-1}}{(n-1)!}\Big)\Big).
\end{align*}
To evaluate the first term, note that $\frac{d H_{\zeta,A(t)}}{dt} = -i_\zeta b$, and so by Lemma \ref{inf action}
\begin{align*}
\frac{d}{dt}\Big|_{t=0} \int_M \mathrm{tr}\Big(\la H_{\zeta,A(t)} \frac{\omega_{A(t)}^n}{n!}\Big)
&=-\int_M\mathrm{tr}\Big(\la i_\zeta b\frac{\omega_A^n}{n!}\Big) + \int_M \mathrm{tr}\Big(\la\Big(dH_{\zeta,A}\wedge b\wedge\frac{\omega_A^{n-1}}{(n-1)!}\Big)\Big)\\
&=\int_M \mathrm{tr}\Big(\la\xi \cdot A\wedge b\wedge\frac{\omega_A^{n-1}}{(n-1)!}\Big)\Big).
\end{align*}
Combining this with \eqref{don} above and Lemma \ref{variation} below completes the proof.
\end{proof}

The following observation can be found in \cite{Mariothesis}, and we reproduce the proof for the convenience of the reader. 

\begin{lem}\label{variation}
Let $A(t)=(A_1(t)=A_1+\sqrt{-1}b_1t,A_2(t)=A_2+\sqrt{-1}b_2t,\ldots)$ be a curve of unitary connections. Then,
\begin{gather}
\frac{d}{dt}\left (\displaystyle \int_M \mathrm{tr}(H_{\zeta,A(t)}) \omega_0^n + \int_M nH_0 \mathrm{tr}(\omega_{A(t)})\wedge  \omega_0 ^{n-1} \right)=0.  
\label{usefulidentityofmario}
\end{gather}
\end{lem}
\begin{proof}
Again using the fact that $\frac{d H_{\zeta,A(t)}}{dt} = -i_\zeta b$ and $\frac{\omega_{A(t)}}{dt} = -db$, we see that
\begin{align*}
\displaystyle \int_M \frac{d}{dt}\mathrm{tr}(H_{\zeta,A(t)}) \omega_0^n &= -\int_M \mathrm{tr}(i_{{\zeta}}b) \omega_0^n  \\
&= \displaystyle \int_M n i_{{\zeta}}\omega_0 \wedge \mathrm{tr}(b) \omega_0 ^{n-1}  \\
&=  -\displaystyle \int_W n dH_{\zeta,0} \wedge \mathrm{tr}(b) \wedge \omega_0 ^{n-1}  \\ 
&= \displaystyle \int_W n H_0\mathrm{tr}(db) \wedge \omega_0 ^{n-1}  \\
&=-\displaystyle \int_W \frac{d}{dt} nH_0 \mathrm{tr}(\omega_{A(t)})\wedge \omega_0 ^{n-1}.
\end{align*}
\end{proof}
\indent To finish the proof of Theorem \ref{thm:momentmapthm}, we observe that given any $(J,A)\in \sM$, $Lie(\tilde \sG)$ can be identified with $ C^\infty(M,\RR)_0\times C^\infty(M,\RR)^{m}$, where the subscript of zero denotes functions with vanishing average with respect to $\omega_0^n$. Indeed, the discussion preceding  Lemma \ref{inf action} shows that given any $\xi \in Lie(\tilde\sG)$, one can associate a tuple $(H_{\zeta,0},H_{\zeta,A_1}, \cdots,H_{\zeta,A_m})$ of smooth functions on $M$. Conversely, given  a tuple $(H_0,H_1,\cdots,H_m)$, we let $\zeta = \nabla_{g_0}H_0$, where the gradient is taken with the respect to the Riemannian metric $g_0 = \omega_0(\cdot,J\cdot)$. Then \eqref{dec} defines a vector field $\tilde\zeta$ on $E$ covering $\zeta$, and defining an element $\xi$ of $Lie(\tilde\sG)$. With this identification, it follows that  $\mu_{(J,A)} \equiv 0$ precisely when 
\begin{align*}
\frac{\omega_{0} ^n}{V_0}&=\frac{\omega_{A_1}^n}{V_1} = \ldots =\frac{\omega_{A_m}^n}{V_m}, \text{ and}\nonumber \\
S_{\omega_0,J} &= tr_{\omega_0} (\omega_0+\omega_{A_1}+\ldots +\omega_{A_m}) + \hat S,
\end{align*}
for some constant $\hat S$. By the Kempf-Ness theorem, as long as a stability condition holds, one expects a zero to occur in the gauge orbit of the complexified Lie algebra action. Akin to the case of the Calabi Conjecture \cite{fine} and the constant scalar curvature K\"ahler equation \cite{Donaldmom, Fujiki}, a zero occurring in the complex gauge orbit is equivalent to varying the metrics $\omega_{A_i}$ in their K\"ahler classes.
\section{Coupled Futaki invariants and $K$-polystability}\label{sec:Futaki} 
 In \cite{Jakobintro, Jakobtorpap}, a Donaldson-Futaki type invariant is defined in the context of coupled K\"ahler-Einstein metrics, using Deligne pairings and intersection-theoretic formulae respectively.  In this section, we introduce an analogue of the differentio-geometric Futaki invariant in the context of these coupled equations, and show that this agrees with the formulae in \cite{Jakobintro, Jakobtorpap}, at least when $W$ is a normal $\QQ$-Fano variety. In fact, we provide twisted versions of these formulae, which we expect will be useful in studying the continuity method for the existence of coupled K\"ahler-Einstein metrics \cite{cKE}. Throughout the section, we let $W$ be a normal variety, and $L_i$ be ample line bundles on $W$ with fixed admissible (say, in the sense of \cite{Ding-Tian}) Hermitian metrics $\psi_i $ with curvature currents $\be_i = \ddbar\psi_i\in 2\pi c_1(L_i)$. We denote a general Hermitian metric on $L_i$ by $e^{-\varphi_i}$ and measure its curvature by $\omega_{\varphi_i} = \ddbar\varphi_i$. Let $\sH_i$ be the space of all positively curved admissible Hermitian metrics on $L_i$. Define $\psi := \psi_0+\cdots+\psi_m$, $\varphi := \varphi_0+\cdots+\varphi_m$  and $\be := \be_0+\cdots+\be_m \in 2\pi c_1(L)$. We denote $$\fg_W := \{w \in H^{0}(W,T^{1,0}W)~|~ w \text{ is Hamiltonian with respect to } \omega_i\in 2\pi c_1(L_i) \text{ for all }i\}$$

For $w\in \mathfrak{g}_{W}$, we denote the Hamiltonian of $w$ with respect to $\omega_{\varphi_i}$ by $\theta_{w,i}$, and let $\theta_w = \theta_{w,0} + \cdots+\theta_{w,m}$. Our convention is that $\theta_{w,i}$ solves $$\sqrt{-1}~\dbar \theta_{w,i} = i_w\omega_{\varphi_i}.$$

\begin{defn} Let $W$ be smooth. The {\em (untwisted) coupled Futaki invariant} is defined as a charater $\cFut(W,(L_i),\cdot):\fg_W \rightarrow \CC$, where 
\begin{align*}
\cFut(W,(L_i),w) &= \sum_{i=1}^m\int_W \theta_{w,i}\Big(\dvvi - \frac{\omega_0^n}{V_0n!}\Big) + \frac{1}{V_0}\int_W\theta_{w,0}(S_{\omega_0} - \hat S - \mathrm{tr}_{\omega_0}\omega)\dvo\\
&= \sum_{i=0}^m\frac{1}{V_i}\int_W \theta_{w,i}\dvj + \frac{1}{V_0}\int_W\theta_{w,0}(S_{\omega_0} - \hat S )\dvo - \frac{1}{V_0}\int_W[\theta_{w,0}\mathrm{tr}_{\omega_0}\omega + \theta_w]\dvo
\end{align*}

\end{defn}
Our first Proposition below shows that $\cFut(W,(L_i),w)$ is an actual invariant of the tuple $(W,(L_i))$, that is, it does not depend on the choice of reference K\"ahler forms in the respective classes $2\pi c_1(L_i)$. 
\begin{prop} \label{fut invariance} $\cFut(W,(L_i),w)$ is independent of the choice of metrics $\omega_i\in 2\pi c_1(L_i)$. In particular, $\cFut(W,(L_i),w)$ vanishes if $(W,(L_i))$ admits coupled cscK metrics. 

\end{prop}

\begin{proof}We consider a family of metrics $\omega_{i,s} = \omega_i + s\ddbar\eta_i$, with corresponding Hamiltonian functions $\theta_{w,i,s} = \theta_{w,i} + sw(\eta_i)$, and denote $\omega_s = \sum_{i=0}^m\omega_{i,s}$ and $\theta_{w,s} = \sum_{i=0}^m\theta_{w,i,s}$. Defining $$f(s) = \sum_{i=1}^m\int_W \theta_{w,i,s}\Big(\dvvjs - \frac{\omega_{0,s}^n}{V_0n!}\Big) + \frac{1}{V_0}\int_W\theta_{w,0,s}(S_{\omega_{0,s}} - \hat S - \mathrm{tr}_{\omega_{0,s}}\omega_s)\dvos,$$ our aim is then to show that $f'(s) = 0$. We rewrite $$f(s) =  \sum_{i=0}^m\int_W \theta_{w,i,s}\dvvjs +  \frac{1}{V_0}\int_W\theta_{w,0,s}(S_{\omega_{0,s}} - \hat S )\dvos - \frac{p(s)}{V_0},$$ where $$p(s) = \int_W\theta_{w,0,s}\mathrm{tr}_{\omega_{0,s}}\omega_s\dvos + \int_W \theta_{w,s}\dvos.$$ It is well known that the first two terms in the expression of $f(s)$ are invariants of the respective K\"ahler classes, and hence it is sufficient to show that $p'(s) = 0$. This is analogous to Lemma \ref{variation}, and indeed is a consequence of it (cf. Remark \ref{rem:futandmoment}). Rather than relying on the moment map interpretation, we  give a direct proof here.   In the computation below, all covariant derivatives are taken with respect to $\omega_{0,s}$, and we also suppress the dependence of the Hamiltonians on $w$. Denoting by $\Lambda_{0,s}$, the contraction by $\omega_{0,s}$, we compute,
\begin{align*}
p'(s) &= \int_W w(\eta_0)\La_{0,s}\omega_s \dvos - \int_W\theta_{0,s}\nabla^{k}\nabla^{\bar l}\eta_0(g_s)_{k\bar l}\dvos + \int_W\theta_{0,s}\Delta_{0,s}\eta \dvos\\
&+\int_W\theta_{0,s}\La_s\omega_s~\Delta_{0,s}\eta_0\dvos + \int_Ww(\eta)\dvos + \int_W\theta_{s}\Delta_{0,s}\eta_0\dvos,\\
\end{align*}

Firstly, for the third term, we notice that  
\begin{align}\label{eq:1}
\int_W\theta_{0,s}\Delta_{0,s}\eta \dvos = -\int_Ww(\eta)\dvos.
\end{align}

Next, integrating the second term by parts, and noting that $\nabla^k\theta_{0,s}(g_s)_{k\bar l} = w^k (g_{s})_{k\bar l} =\nabla_{\bar l}\theta_{s} $,
\begin{align}\label{eq:2}
- \int_W\theta_{0,s}\nabla^{k}\nabla^{\bar l}\eta_0(g_s)_{k\bar l}\dvos =- \int_W\theta_s\Delta_{0,s}\eta_0\dvos + \int_W\theta_{0,s}\nabla^{\bar l}\eta_0\nabla_{k}(g_s)_{k\bar l}\dvos,
\end{align}
where we integrated by parts a second time in the first term. Now it is easy to see  (for instance by using normal coordinates for $\omega_{0,s}$) that, $\nabla_{k}(g_s)_{k\bar l} = \nabla_{\bar l}\mathrm{tr}_{\omega_{0,s}}\omega_s$, and so 
\begin{align*}
\int_W\theta_{0,s}\nabla^{\bar l}\eta_0\nabla_{k}(g_s)_{k\bar l}\dvos &= \int_W\theta_{0,s}\nabla^{\bar l}\eta_0\nabla_{\bar l}\mathrm{tr}_{\omega_{0,s}}\omega_s\dvos\\ 
&=  -\int_W w(\eta_0) \La_{{0,s}}\omega_s\dvos - \int_W \theta_{0,s}\Delta_{0,s}\eta_0 \La_{0,s}\omega_s\dvos,
\end{align*}
and so 
\begin{align}\label{eq:3}
- \int_W\theta_{0,s}\nabla^{k}\nabla^{\bar l}\eta_0(g_s)_{k\bar l}\dvos =- \int_W\theta_s\Delta_{0,s}\eta_0\dvos  &- \int_W w(\eta_0) \La_{{0,s}}\omega_s\dvos\\& - \int_W \theta_{0,s}\Delta_{0,s}\eta_0 \La_{0,s}\omega_s\dvos. \nonumber
\end{align}
Then combining \eqref{eq:1} and \eqref{eq:3} we see that $p'(s) = 0.$

\end{proof}
 
\begin{rem} [Futaki invariant and the moment map formalism]\label{rem:futandmoment}
The Futaki invariant is essentially the moment map from the previous section, evaluated on a certain subspace of $Lie(\mathcal{G}^\CC)$. More precisely, let $K \subset \mathcal{G}$ be the stabilizer of a tuple $(A_1,\cdots,A_m,J) \in \sM$. Then $K$ is a finite dimensional compact Lie group, and hence has a complexification, which we denote by $K^\CC$. Note also, that by Lemma \ref{inf action}, since $K$ is a stabilizer for $(J,A)$, each $\tilde w \in K$ covers a vector field $w$ which is Hamiltonian with respect to each of $\omega_i = \omega_{A_i}$, and so we can identify $\mathrm{Lie}(K^\CC)$ as a subspace of $\fg_W$. Then it is easy to see that for any $w\in \mathrm{Lie}(K^\CC)$,$$\cFut(W,w) = \mu_{(J,A)}(\tilde w),$$ if we normalize $\theta_{w,0}$ to have zero mean with respect to $\dvo$. Next, observe that moving $(\omega_0,\cdots,\omega_m)$ in their K\"ahler classes is equivalent to the action of $K^\CC$ on $(J,A)$, and then Proposition \ref{fut invariance} is simply the formal statement that if $\tilde g_t\in K^\CC$, then $$\frac{d}{dt}[\mu_{\tilde g_t\cdot(J,A)}(\mathrm{ad}_{\tilde g_t}(\tilde w))] = 0,$$ since $\mathrm{ad}_{\tilde g_t}w$ lies in the complexification of the stabilizer of $g_t\cdot(J,A)$, and thus corresponds to the zero vector in $\mathcal{T}_{\tilde g_t\cdot (J,A)}\sM$.  Also, note that with this formalism, $p'(s) = 0$ in the proof above is essentially equivalent to Lemma \ref{variation}.
\end{rem}

\subsection{Donaldson-Futaki invariants and K-polystability}\label{subsec:DF} In order to define K-polystability, it is necessary to extend the above definition of the coupled Futaki invariant to possibly singular varieties $W$. While one can probably extend the techniques in \cite{Ding-Tian} to achieve this objective, following Donaldson \cite{Do}, we instead prove an alternate algebro-geometric formula, which in turn can be used as a definition of the coupled Futaki invariant on singular varieites. So consider a smooth polarized tuple $(W,(L_i))$ as before, but now with a $\CC^*$ action on each total space $L_i$, covering a fixed $\CC^*$ action on $W$ generated by a holomorphic vector field $w$. By the Riemann-Roch theorem, the dimensions $d_{i,k}$ of $H^0(M,L_i^k)$ satisfy the expansion $$(2\pi)^nd_{i,k} = a_{i,0}k^n + a_{i,1}k^{n-1} + O(k^{n-2}),$$ where 
\begin{equation}\label{a}
a_{i,0} = \int_W\dvj = (2\pi)^n\frac{L_i^n}{n!},~a_{i,1} = \frac{1}{2}\int_WS_{\omega_i}\dvj = (2\pi)^n\frac{(-K_W)\cdot L_i^n}{2(n-1)!}.
\end{equation}
Next, note that if $\hat w_i$ is the vector field generated by the $\CC^*$ action on $L_i$, then just as in \eqref{dec}, one can write $$\tilde w_i = w_i^{hor} - \sqrt{-1}\theta_{w,i}\underline t,$$ where $ w_i^{hor}$ is the horizontal lift of $w$ with respect to the Chern connection of $e^{-\varphi_i}$ (or equivalently the Levi-Civita connection of $\omega_i$), and $\underline{t}$ is the canonical vertical vector field on $L_i$.  The fact that $\hat w_i$ is holomorphic is then precisely the condition that $\theta_{w,i}$ is the Hamiltonian of $w$ with respect to $\omega_i$. Now if $w_{i,k}$ is the total weight of the action on $H^0(W,L_i^k)$, then by the equivariant Riemann-Roch theorem, $$(2\pi)^nw_{i,k} = b_{i,0}k^{n+1} + b_{i,1}k^{n} + O(k^{n-1}),$$ where 
\begin{equation}\label{b}
b_{i,0} = -\int_W\theta_{w,i}\dvj,~b_{i,1} = -\frac{1}{2}\int_WS_{\omega_i}\theta_{w,i}\dvj. 
\end{equation}
We denote the coefficients corresponding to $L$ by $a_0,a_1,b_0$ and $b_1$. Additionally, we also need to consider the space of sections of $L_0^k\otimes L^{-1}$, and we denote the corresponding dimension and weight by $d_{t,k}$ and $w_{t,k}$ respectively. Then by Corollary \ref{twisted expansion} in the Appendix we have 
\begin{align*}
(2\pi)^nd_{t,k} = a_{t,0}k^{n} + a_{t,1}k^{n-1} + O(k^{n-2})\\
(2\pi)^nw_{t,k} = b_{t,0}k^{n+1} + b_{t,1}k^{n} + O(k^{n-1}), 
\end{align*} where 
\begin{align}\label{twisted a b}
a_{t,0} &=\int_W\dvo,~ a_{t,1} = \int_W\Big(\frac{S_{\omega_0}}{2} - \mathrm{tr}_{\omega_0}\omega\Big)\dvo\\
b_{t,0} &= -\int_W\theta_0\dvo,~ b_{t,1} = -\frac{1}{2}\int_W\theta_0S_{\omega_0}\dvo + \int_W[\theta_0\mathrm{tr}_{\omega_0}\omega + \theta]\dvo.\nonumber
\end{align}
A simple computation now proves the following 
\begin{prop}\label{DF}
Suppose there is a $\CC^*$ action on $(W,(L_i))$ covering a $\CC^*$ action on $W$ generated by a holomorphic vector field $w$, then $$\cFut(W,(L_i),w) =  \Fut(W,L_0,w) -\sum_{i=0}^m\frac{b_{i,0}}{a_{i,0}} +  \frac{(a_{t,1}b_{t,0} - a_{t,0}b_{t,1})}{a_{t,0}^2},$$ where $\Fut(W,L_0,w)$ is classical Futaki invariant for the class $L_0$ given by $$\Fut(W,L_0,w) = \frac{(a_{0,1}b_{0,0} - a_{0,0}b_{0,1})}{a_{0,0}^2}.$$
\end{prop}

\begin{rem}\label{Ruis rem}
Our formula above is analogous to the formula for the Donaldson-Futaki invariant obtained in \cite{Der} in the context of twisted cscK metrics. Moreover it is also shown in that paper (cf.\ \cite[Lemma 2.30]{Der}) that when $W$ is an arbitrary (possibly non-smooth) normal variety, the dimensions $d_{t,k}$ and the weights $w_{t,k}$ are given by polynomials of degrees $k^n$ and $k^{n+1}$ respectively. Based on this, we can then use the right hand side of the formula above to {\em define} the coupled Futaki invariant for normal varieties. 
\end{rem}
In view of the above remark, we can now finally define Donaldson-Futaki invariants for test configurations of polarized tuples with normal central fibres and the relvant notion of K-polystability. 
\begin{defn} A test configuration (with exponent $r$) for $(W,L_0,L_1,\cdots,L_m)$, with $L = \otimes L_i$ as above, consists of a normal variety $\sW$ polarized by a tuple $(\sL_0,\cdots,\sL_m)$ with the following additional data : 
 \begin{enumerate}
 \item A $\CC^*$ action on $\sW$ lifting to actions on $(\sL_0,\cdots,\sL_m)$.
 \item A flat $\CC^*$ equivariant map $\pi:\sW\rightarrow \PP^1$, where $\PP^1$ has the standard $\CC^*$ action, such that $(\pi^{-1}(\PP^1\setminus\{0\},\sL_0,\cdots,\sL_m, \otimes_{i=0}^m \sL_i)$ is equivariantly isomorphic to $(W\times \CC,p_W^*L^r_0,\cdots,p_W^*L_m^r,p_W^*L^r)$.
 \end{enumerate}
 A test configuration is called a {\em product}, if $(\sW,\sL_0,\cdots,\sL_m, \otimes_{i=1}^m \sL_i)$ is equivariantly isomorphic to $(W\times \PP^1,p_W^*L^r_0,\cdots,p_W^*L_m^r,p_W^*L^r)$. It is called a {\em smooth degeneration} if $\mathcal{W}_0$ is non-singular. 
 
 \end{defn}
 
 \begin{defn} The Donaldson-Futaki invariant for a test configuration $(\sW, (\sL_i))$ of $(W,(L_i))$ is defined as  $$\DF(\sW,(\sL_i)):= \cFut(\sW_0,(\sL_{i,0}),w),$$ where $\sL_{i,0} := \sL_{i}\Big|_{\pi^{-1}(0)}$ and $w$ is the vector field generating the induced $\CC^*$ action on $\sW_0$. In the event that $\sW_0$ is singular, the coupled Futaki invariant is defined simply by the right side in Proposition \ref{DF} (cf.\ Remark \ref{Ruis rem}). \end{defn}
 
 \begin{defn}
 A tuple $(W,(L_i))$ is said to be said to be {\em semistable} if for any test configuration $(\sW,(\sL_i))$, $$\DF(\sW,(\sL_i))\geq 0.$$ It is said to be {\em polystable} if, in addition to being semistable, the Donaldson-Futaki invariant vanishes if and only if $(\sW,(\sL_i))$ is a product.
 \end{defn}

\subsection{Coupled Donaldson-Futaki invariants on Fano manifolds} \label{Fano} We now specialize to the case when $W$ is a normal $\QQ$-Fano variety, and $L = K_W^{-r}$. Recall that $K_W$ is well defined as a Weil divisor, and $K_W^{-r}$ extends as an ample line bundle on $W$ for some $r\in \NN$. For simplicity of notation, we assume that $r=1$ throughout. We then call such a polarized tuple  $(W,(L_i))$, a {\em  polarized Fano tuple}. Note that $e^{-\varphi}$ and $e^{-\psi}$ are now Hermitian metrics on $K_W^{-1}$, and hence are volume forms on $W$, and also that $\hat S = 0$. 

In \cite{Jakobtorpap}, an intersection-theoretic definition of the Donaldson-Futaki invariant is given in the context of test configurations for Fano tuples.  Our next aim is to show (cf.\ Theorem \ref{equivalence}) that in this special case, our formula for the Donaldson-Futaki invariant agrees with the one in \cite{Jakobtorpap}.    As a first step towards proving Theorem \ref{equivalence}, we obtain a much simpler formula for  the coupled Futaki invariant on Fano tuples analogous to the formula in  \cite{He} for the classical Futaki invariant. An advantage is that even though it is an integral formula (as opposed to an algebro-geometric one), it is much more transparently well defined on possibly singular normal varieties. While this paper was in preparation, an analogous formula appeared in the context of coupled Sasaki-Einstein metrics in \cite{Futakicoupled}.

\begin{lem}\label{Fano Futaki}
If $(W,(L_i))$ be a smooth polarized Fano tuple as above, then $$\cFut(W,w) = \sum_{i=0}^m\frac{1}{V_i}\int_W\theta_{w,i}\dvj - \frac{\int_W\theta_we^{-\varphi}}{\int_We^{-\varphi}}.$$
\end{lem}

\begin{proof}
Note that $\hat S = 0$ since $L = -K_W$. If $h_0\in C^{\infty}(W,\RR)$ such that $$\mathrm{Ric}(\omega_0) - \omega = \ddbar h_0,$$ then 
\begin{align*}
\cFut(W,w) &= \sum_{i=0}^m\frac{1}{V_i}\int_W\theta_{w,i}\dvj + \frac{1}{V_i}\int_W\theta_{w,0}\Delta_{\omega_0}h_0\dvo - \frac{1}{V_0}\int_W\theta_w\dvo,\\
&= \sum_{i=0}^m\frac{1}{V_i}\int_W\theta_{w,i}\dvj - \frac{1}{V_0}\int_W[w(h_0)+\theta_w]\dvo.
\end{align*}
A simply computation (for instance by taking $\dbar$ on both sides) shows that $$\Delta_{\omega_0}\theta_{w,0} + \theta_w + w(h) = c$$ for some constant $c$, and hence $$ \frac{1}{V_0}\int_W[w(h_0)+\theta_w]\dvo = c.$$  So it is enough to evaluate the constant $c$. Integrating, with respect to $e^{h_0}\omega_0^n$ we see that 
\begin{align*}
cV_0 &= \int_W [\Delta_{\omega_0}\theta_{w,0} + w(h_0)]e^{h_0}\dvo + \int_W\theta_we^{h_0}\dvo\\
&= \int_W\theta_we^{h_0}\dvo. 
\end{align*}
Next, from the definition of $h_0$, it is easy to see that $e^{h_0}\omega_0^n = be^{-\varphi}$ for some constant $b$, which by integration can be found to be $$b= \frac{V_0}{\int_We^{-\varphi}}.$$ Together with the formula for $c$ above, we see that $$c = \frac{\int_W \theta_we^{-\varphi}}{\int_W e^{-\varphi}}.$$
\end{proof}

 \begin{prop}\label{equivalence} Let $(\sW,(\sL_i))$ be a test configuration with a smooth central fiber $(W,(L_i))$, and let $w$ be the induced holomorphic vector field on W. Then $$\DF(\sW,(\sL_i)) = -\frac{1}{n+1}\sum_{i=1}^m\frac{\sL_i^{n+1}}{L_i^n} + \frac{(K_{\sW/\PP^1} +\sum_j\sL_i)\cdot(\sum\sL_i)^n}{(-K_W)^n}$$
 %Let $(W,(L_i))$ be as before, and let $w\in \mathfrak{g}_{W,\psi}$. Let $(\sW,(\sL_i))$ be the associated test configuration, as in example \ref{ex:ass tc}. Then $$ \DF(\sW,(\sL_i)) = (n+1)\cFut(W,w).$$
 \end{prop}

 \begin{proof}
 In the classical case of K\"ahler-Einstein metrics, such intersection formulae were first obtained in \cite{Wa} and \cite{Odaka}. We instead follow the exposition in \cite{Li-Xu}.  Let $(\sW,\Nu)$ be any test configuration (so $m= 0$ in the above definition). The $\CC^*$ action induces an action on the total space $\Nu_0$ covering a $\CC^*$ action on $\sW_0$, which we assume is generated by the vector field $w$. By the Riemann-Roch theorem, $$\mathrm{dim} H^0(\sW_0,\Nu_0) = a_0(\Nu_0)k^n + a_1(\Nu_0)k^{n-1} + O(k^{n-2}). $$Now, $\CC^*$ acts on $\Nu_0$, so if $w_k(\Nu_0)$ is the total weight of the $\CC^*$ action on $H^0(\sW_0, \Nu_0)$, by the equivariant Riemann-Roch theorem, $$w_k(\Nu_0) = b_0(\Nu_0)k^{n+1} + b_1(\nu_0)k^n + O(k^{n-1}),$$ where $n$ is the dimension of $\sW_0$. Then it is shown in \cite{Li-Xu}, that  
 
  \begin{align}\label{li-xu formulae}
 b_0(\Nu_0) &= \frac{\Nu^{n+1}}{(n+1)!}\\
 b_1(\Nu_0) &= \frac{1}{2}\frac{(-K_\sW)\cdot \Nu^n}{n!} - a_0(\Nu_0).\nonumber
 \end{align}
 
 On the other hand, given any Hermitian metric $e^{-\nu}$ with positive curvature form $-\sqrt{-1}\Omega$ on $\Nu_0$, the $\CC^*$ action on the total space $\Nu_0$ induces a Hamiltonian $H$ for $w$ with respect to $\Omega$ and similar to similar to formulae \eqref{a}-\eqref{b}, we have
 
 \begin{align}\label{coefficients}
a_0(\Nu_0) = (2\pi)^{-n}\int_{\sW}\frac{\Omega^N}{N!} = \frac{\Nu_0^n}{n!},~a_1(\Nu_0) = \frac{(2\pi)^{-n}}{2}\int_{\sW}S_{\Omega}\frac{\Omega^N}{N!}\\
b_0(\Nu_0) = -(2\pi)^{-n}\int_{\sW}H\frac{\Omega^N}{N!},~b_1(\Nu_0) = -\frac{(2\pi)^{-n}}{2}\int_\sW HS_\Omega\frac{\Omega^N}{N!}\nonumber.
\end{align}

Applying the second formula in \eqref{li-xu formulae} to $\Nu = \sL$ and $\Nu_0 = -K_{W}$, we see that 
 \begin{align*}
 K_{\sW/\PP^1}\cdot\sL^n &= K_{\sW}\cdot \sL^n - \pi^*K_{\PP^1}\cdot\sL^n\\
 &= -2n!b_1(-K_W).
 \end{align*}
Now applying the first formula in \eqref{li-xu formulae} and the formulae in \eqref{coefficients} to $\Nu \in \{\sL_1,\cdots,\sL_m,\sL\}$, we have 
 \begin{align}\label{first step}
-\frac{1}{n+1}\sum_{i=0}^m\frac{\sL_i^{n+1}}{L_i^n} &=  \sum_{i=0}^m\frac{1}{V_i} \int_W \theta_{w,i}\dvj \\
 \frac{(K_{\sW/\PP^1} +\sum_j\sL_i)\cdot(\sum\sL_i)^n}{(-K_W)^n} &=\frac{1}{(-K_W)^n} \Bigg ( \displaystyle n\int_W {\theta} c_1 ({\omega})\wedge {\omega}^{n-1}- (n+1) \int_W {\theta}\omega^n\Bigg ) \nonumber
%&=  \sum_j (n+1)\int_W \frac{\theta_i \omega_i^n}{V_i} + \frac{n+1}{(-K_W)^n} \Bigg ( -\int_W {\theta} {\omega}^n +n\int_W {\theta} \partial \bar{\partial} {h_\omega} \wedge {\omega}^{n-1} \Bigg),
\end{align}
For the second equation if we let $h_\omega$ be the Ricci potential of $\omega$, that is $$\mathrm{Ric}(\omega) - \omega = \ddbar h_\omega,$$ then,
\begin{align}\label{second step}
\frac{(K_{\sW/\PP^1} +\sum_i\sL_i)\cdot(\sum\sL_i)^n}{(-K_W)^n}  &= \frac{1}{(-K_W)^n} \Bigg ( -\int_W {\theta} {\omega}^n +n\int_W {\theta} \partial \bar{\partial} {h_\omega} \wedge {\omega}^{n-1} \Bigg)\\
&= \frac{1}{(-K_W)^n} \Bigg ( -\int_W {\theta} {\omega}^n -\int_W w({h_\omega}) {\omega}^{n} \Bigg)  \nonumber\\
&= -\frac{\int_W\theta \,e^{-\varphi}}{\int_W e^{-\varphi}},\nonumber
\end{align} 
where we have used the well known fact that $\theta_w$ satisfies $$\Delta_\omega \theta_w + \theta_w + w(h) =\frac{\int_W\theta \,e^{-\varphi}}{\int_W e^{-\varphi}}.$$ Combining \eqref{first step} and \eqref{second step} with Lemma \ref{Fano Futaki} completes the proof of the proposition. \end{proof}
\begin{rem}
Proposition \ref{equivalence}, and the formula in Lemma \ref{Fano Futaki}, in all likelihood also hold when $W$ is a $\QQ$-Fano normal variety. To prove this, one would have to show that the formulae \eqref{twisted a b} for the coefficients of the twisted weights also hold in this generality.  This can probably be done by using the equivariant Riemann-Roch theorem (cf. \cite{Do}) to calculate the coefficients that appear in Lemma 2.30 in \cite{Der}.
\end{rem}

\subsection{An aside: Twisted coupled K\"ahler-Einstein metrics} We continue using the notation of subsections \ref{subsec:DF} and \ref{Fano} above. In particular, recall that $e^{-\psi_i}$ is a continuous metric on $L_i$ with curvature $-\sqrt{-1}\be_i$, and we let $\psi + \psi_0+\cdots+\psi_m$.

\begin{defn} Twisted coupled K\"ahler-Einstein metrics  on $(W,(1-t)\psi, (L_i))$ are a tuple 
$(e^{-\varphi_0},\cdots,e^{-\varphi_m})$ of positively curved Hermitian metrics on $(L_1,\cdots,L_m)$ solving $$\frac{\omega_{\varphi_0}^n}{V_0} = \cdots=\frac{\omega_{\varphi_m}^n}{V_m} = \frac{e^{-t\varphi-(1-t)\psi}}{\int_W e^{-t\varphi - (1-t)\psi}},$$ where $V_i =  (2\pi L_i)^n/n!$ is the volume of $W$ with respect to the class $ 2\pi c_1(L_i)$.

\end{defn}

The twisted coupled Futaki invariant is defined as a character on the restricted Lie algebra $$\mathfrak{g}_{W,\psi} :=\{w\in {\rm H^0(W,TW)}~|~w \text{ generates a $\CC^*$ action, and } i_w \be_i = 0,~\text{for all $i = 0,1,\cdots,m$}\}.$$ 
 
\begin{defn} \label{futakidefinition} 

The {\em twisted coupled Futaki invariant} is defined by $$\ctFut(W,w) := \cFut(W,w) -(1-t)\sum_{i=0}^m\frac{1}{V_i}\int_W\theta_{w,i}(\beta_i - \omega_{\varphi_{i}})\wedge \frac{\omega_{\varphi_i}^{n-1}}{(n-1)!}$$

\end{defn}

The next proposition shows that the above formulae do define invariants of the respective K\"ahler classes. 

\begin{prop}Let $\varphi_{i,s} = \varphi_i+ s\eta_i$, $\omega_{i,s} = \ddbar\varphi_{i,s}$, $\theta_{w,i,s}$ the corresponding Hamiltonians, and define $$f(s) =\sum_{i=0}^m\frac{1}{V_i}\int_W\theta_{w,i,s}\dvjs - \frac{\int_W\theta_w e^{-\varphi_s}}{\int_W e^{-\varphi_s}}-(1-t)\sum_{i=0}^m\frac{1}{V_i}\int_W\theta_{w,i,s}(\beta_i - \omega_{{i,s}})\wedge \frac{\omega_{{i,s}}^{n-1}}{(n-1)!}.$$  Then $f'(s) = 0$.

\end{prop}

\begin{proof}
This proposition follows easily by rewriting the coupled Futaki invariant in terms of the classical Futaki invariant. It is a standard fact that for any K\"ahler metric $\omega_s = \nddbar\varphi_{s} \in c_1(W)$ and any holomorphic vector field $w$ with Hamiltonian $\theta_{w,s}$, 
\begin{align}\label{key identity}
\Delta \theta_{w,s} + \theta_{w,s} + w(h_{\omega_{w,s}}) = \frac{\int_W\theta_{w,s}e^{-\varphi_s}}{\int_We^{-\varphi_s}},
\end{align} and so 
\begin{align*}
f(s) = \Fut(W,w) &+ \sum_{i=0}^m\frac{1}{V_i}\int_W\theta_{w,i,s}\dvjs - \frac{1}{V}\int_W\theta_{w,s}\dvols \\ &-(1-t)\sum_{i=0}^m\frac{1}{V_i}\int_W\theta_{w,i,s}(\beta_i - \omega_{{i,s}})\wedge \frac{\omega_{{i,s}}^{n-1}}{(n-1)!}.
\end{align*}
 Now the last three terms are clearly invariants of the K\"ahler class, as can be seen by differentiating them.
 \end{proof}

%it is not even clear that these define invariants of $(W,L_1,\cdots,L_m)$. But this follows immediately from the equivalent algebro-geometric formulae in Proposition \cite{RR_futaki} below, which in turn is an extension of the formula of Donaldson for the classical Futaki invariant \cite{Donaldson} (cf.\ \cite{Dervan} for the twisted variant). 

\section{Perturbing coupled cscK metrics}\label{sec:perturbsection} 
\indent We prove Theorem \ref{thm:perturb} in this section. In what follows, we assume that $(M,J, (L_i))$ admit coupled cscK metrics $(\omega_0,\cdots,\omega_m)$. We fix hermitian metrics $h_i$ on $L_i$ with Chern connection $A_i$ and curvature $-\sqrt{-1}\omega_i$ to obtain a point $(J,A_1,\cdots,A_m)\in \sM$. Recall that in section \ref{sec:moment} we interpreted coupled cscK metrics as zeros of a moment map $\mu:\sM\rightarrow \mathrm{Lie}(\tilde G)$, and so $\mu(J,A) = 0$. Even though the gauge group $\sG$ might not have a complexification, following the ideas in \cite{Donaldmom}, a key point is that one can make sense of the orbits of such a complexification. First, note that $\mathrm{Lie}(\tilde \sG)$ has a Lie algebra complexification $\mathrm{Lie}(\tilde \sG)^\CC$ and the infinitesimal action from Lemma \ref{inf action} has an obvious extension which we still denote by $$P:\mathrm{Lie}(\tilde\sG)^\CC\rightarrow\mathcal{T}_{(J,A)}\sM.$$  We then say that $(A_0,J_0)$ and $(A_1,J_1)$ are in the same $\mathrm{\tilde \sG}^\CC$ orbit if there is a path $(A_t,J_t)$ in $\sM$ and a path $\xi_t\in \mathrm{Lie}(\tilde\sG)^\CC$ such that for all $t\in [0,1]$, $$\frac{d}{dt}(A_t,J_t) = P(\xi_t).$$ The basic ideas of Br\"onnle \cite{Bron} and Sz\'ekelyhidi \cite{Gabor}, with small modifications due to \cite{Inoue, Rui} can now be summarized as follows.
\begin{enumerate}
\item Following Kuranishi \cite{Kur}, one constructs a holomorphic slice $\Phi : B_1 \rightarrow \mathcal{M}$ such that $\Phi(x)$ meets the $\tilde\sG^\CC$ orbit of every $J'$ sufficiently close to $J$. Here $B_1$ is a small ball (whose size is decided as one goes along the proof) in a finite dimensional space $\tilde H^1$ ``normal" to the action of the complex gauge group. The finite dimensional subgroup $K\subset \tilde\sG$ stabilizing $(J,A)$ has a legitimate complexification $K^\CC$ and also has a natural action on $\tilde H^1$. One can then ensure that whenever $x, x^{'} \in B$ lie in a $K^\CC$-orbit, then $\Phi(x), \Phi(x^{'})$ lie in the same complex infinite-dimensional gauge orbit.
\item Using the implicit function theorem, one can perturb the image of $\Phi$ within the same complex gauge orbit so that the infinite-dimensional moment map $\mu(\Phi(x)) \in \mathfrak{k}$ where $\mathfrak{k}$ is the Lie algebra of $K$.
\item One pulls back the GIT problem from $\mathcal{M}$ to $\tilde{H}^1$, i.e., one considers $\tilde{H}^1$ with the symplectic form $\Phi_{x=0}^{*}\Omega$ and the linear Hamiltonian action of $K$. By the finite-dimensional Kempf-Ness theorem, there exists a zero of the corresponding moment map $\nu$ at a vector $v_0 \in \tilde{H}^1$ in any polystable gauge orbit. It is easy to see that such a  $v_0 \in B$.
\item Thanks to a small generalisation of an observation of Donaldson (proposition 9 in \cite{Gabor}), as long as the derivative of the moment map $\mu$ is uniformly invertible in a neighbourhood, and $\mu(x_0)$ is small, one can perturb $x_0$ to $y$ such that $\mu(y)=0$. A calculation shows that $\mu(\Phi(tv_0)) = O(t^3)$. A few estimates then show that the assumptions of Donaldson's lemma are satisfied and hence there exists a $v \in  B$ such that $\mu(\Phi(v))=0$ whenever $v$ is polystable with respect to $K^\CC$.
\item This step is due to \cite{Rui, Inoue} and it fills a possible gap in the proof of \cite{Gabor}. There exists a $v\in B$ so that the slightly deformed tuple that is under consideration is $\Phi(v)$. If $v$ is $K^{\CC}$-polystable, we are done by the previous steps. If it is either strictly semistable or unstable, the Hilbert-Mumford criterion implies that a limiting object is a zero of the finite-dimensional moment map $\nu$ (in the unstable case, $0$ is the limit). Using the previous steps we produce a coupled cscK metric on the limiting object (in case the limit is $0$, it already has a coupled cscK metric by assumption). Using a construction of a test configuration due to \cite{Gabor}, we see that  K-polystability implies that the limiting object is biholomorphic to $(M,J',A')$ and thus we have a coupled cscK metric on $(M,J',A')$.
\end{enumerate}
 \indent The first step, i.e., constructing a slice, is accomplished by using proposition 3 of \cite{Mariotipler}. 
\begin{rem}
Strictly speaking, we need to work with $\mathcal{N}_k$, the completion of $\mathcal{N}$ in the $H^k$ norm, instead of $\mathcal{N}$. However, firstly, we can choose a sufficiently large $k$ so that the resulting objects are highly differentiable. And at the end of the day, we aim to produce a $C^l$-smooth tuple satisfying the coupled cscK equations. If  $l$ is sufficiently large, which can be ensured by choosing $k$ sufficiently large to begin with, elliptic regularity ensures smoothness of the $C^l$ solutions. Secondly,  the relevant Hilbert manifolds on which we apply the implicit function theorem are carefully spelt out in \cite{Inoue}. For the sake of clarity in exposition, we work with $\mathcal{N}$ just as in \cite{Gabor}. 
\label{hilbertmanifold}
\end{rem}
 Since we restrict ourselves to integrable tuples, we consider the following maps $\tilde{\partial}_i : T_{(J,A)}\mathcal{N} \rightarrow \Omega^{0,2}(T^{1,0})\times \Omega^{(0,2)}$ given by $$\tilde{\partial}_i(T,b) =\left(\bar{\partial}_JT,\bar{\partial}_{(J,A)}b+\frac{\sqrt{-1}}{2}F_A^{T}\right),$$ where recall that $T$ is a section of $TM\otimes T^*M$, and $F_A^T$ is defined simply by contracting $F_A$ with the $TM$ part of $T$. These maps detect whether infinitesimal deformations of $(J,A)$ are integrable or not. By repeated applications of proposition 2 of \cite{Mariotipler}, it can be easily proven that the complex
$$Lie(\tilde{\mathcal{G}})^\CC\xrightarrow{P} T_{(J,A)}\mathcal{N} \xrightarrow{\tilde{\partial}_1 \oplus \tilde{\partial}_2\ldots} \Omega^{0,2}(T^{1,0})\times\Omega^{0,2}\times \Omega^{0,2}\ldots $$ is an elliptic complex. Denote by $\tilde{\partial}$ the map $\tilde{\partial}_1\oplus \tilde{\partial}_2\ldots$. Let $\tilde{H}^1$ be the subspace $\tilde{H}^1=ker(\Delta)\subset T_{(J,A_1,\ldots)}\mathcal{N}$ where $\Delta = PP^{*}+\tilde{\partial}^{*} \tilde{\partial}$. Note that $\tilde{H}^1$ consists of infinitesimal integrable deformations that are orthogonal to the complex gauge orbit. Let $K\subset \tilde{\mathcal{G}}$ denote the stabilizer of $(J,A)$ and $\mathfrak{k}$ its Lie algebra. Then $K$ is a finite-dimensional Lie group, and the kernel of $P$ can be identified with $\mathfrak{k}$. Denote by $K^\CC$ the complexification of $K$. We can now complete steps 1 and 2 in the above strategy.
\begin{prop}\label{kuranishi}
There exists a small ball centred at the origin $B\subset \tilde{H}^1$ and a map $\Phi : B \rightarrow \mathcal{N}$ such that 
\begin{enumerate}
\item $\Phi$ is $K$-equivariant, holomorphic, and $\Phi(0)=(J,A_1,\ldots,A_m)$.
\item The $\tilde{\mathcal{G}}^\CC$ orbit of every integrable almost complex structure $J'$ near $J$ intersects intersects the image of $\Phi$.
\item If $x,x'$ are in the same $K^\CC$-orbit and $\Phi(x)$ is integrable, then $\Phi(x), \Phi(x')$ are in the same $\tilde{\mathcal{G}}^\CC$-orbit.
\item $\mu(\Phi(x)) \in \mathfrak{k} \ \forall \ x \in B$, where $\mu$ is the moment map in theorem \ref{momentmapthm}. (We assume that Lie algebras are identified with their duals using a metric.)
\end{enumerate}
\label{firststepslice}
\end{prop} 
\begin{proof}
An application of proposition 3 to each of the line bundles $L_i$ yields a map $$\Phi_{1} : B_1 \rightarrow \mathcal{N}$$ satisfying all the requirements except the last. Since $\mu(\Phi_{1}(0))=0 \in \mathfrak{k}$, just like in \cite{Gabor}, one can hope to perturb $\Phi_{1}$ within a complex gauge orbit to get a $\Phi$ so that $\mu(\Phi(x))\in \mathfrak{k}$.  \\
\indent Denote by $\mathfrak{k}_{u,l}^{\perp}$ the Sobolev space (where $u,l\gg 1$) of tuples of elements $\tilde{g}=(\sqrt{-1}g_1,\sqrt{-1}g_2,\ldots,\sqrt{-1}g_m, H_0)$ such that they are $L^2$-perpendicular to $\mathfrak{k}$ and $g_i \in H^s \ \forall \ i, H_0 \in H^l$. We identify $Lie(\tilde{\mathcal{G}})$ with $Lie(\mathcal{G})\times Lie(\mathcal{H})$ treating them purely as vector spaces. Let $U_{l,l} \in \mathfrak{k}_{l,l} ^{\perp}$ be a small ball around the origin. Consider the map $G : B_1 \times U_{l,l} \rightarrow \mathfrak{k}^{\perp}_{l-2,l-4}$
\begin{gather}
G(x,\tilde{g})=(\sqrt{-1}g_1,\ldots,\sqrt{-1}g_m,H_0))=\mu^{\perp}(F_{\tilde{g}}(\Phi_1(x))),
\label{mainmap}
\end{gather}
where $F_{\tilde{g}} : \mathcal{N} \rightarrow \mathcal{N}$ is obtained by the unit time flow of the infinite-dimensional vector field induced by $\tilde{g}$, i.e., $F_{\tilde{g}}(J,A_1,A_2,\ldots) = (J(1),A_1(1),\ldots) $ where $$\frac{d(J(t), A_1(t),A_2(t)\ldots)}{dt} = P_{J(t),A(t)}(\tilde{g}).$$ 
Denote $(J(t), A_1(t),\ldots, A_m(t))$ by $w(t)$.We now use the implicit function theorem on Hilbert manifolds to prove that $\tilde{g}$ can be solved for smoothly in terms of $x$ so that $G(x,\tilde{g}(x))=0$ near $x=0$. To this end, we need to prove that $D_{\tilde{g}}G(0)$ is an isomorphism, i.e., there is no vector $v\in \mathfrak{k}^{\perp}_{l,l}$ such that $\frac{dG}{ds} \vert_{s=0} =0$ where $\frac{d\tilde{g}(s)}{ds} \vert_{s=0} =v$. Indeed,
\begin{gather}
0=\frac{dG}{ds} \vert_{s=0} = d\mu^{\perp}_{J,A_1,A_2,\ldots}\left(\frac{dF_{\tilde{g}}}{ds}\vert_{s=0}\right) \nonumber \\
\Rightarrow d\langle \zeta, \mu_{w(0)} \rangle \left(\frac{dF_{\tilde{g}}}{ds}\vert_{s=0}\right) = 0 \ \forall \ \zeta \in \mathfrak{k}^{\perp}.  
\label{infchange}
\end{gather} 
Note that 
\begin{gather}
\frac{d(dw/ds)\vert_{s=0}}{dt} = P_{w(0,t)}\left(\frac{d\tilde{g}}{ds} \vert_{s=0}\right)=P_{w(0)}(v),
\end{gather}
where the last equality holds because at $s=0$, $w(t)=w(0) \ \forall \ t$. Since $F_{\tilde{g}(s)}(J,A_1,A_2,\ldots)= w(s,1)$, we see that 
\begin{gather}
\frac{dF_{\tilde{g}}}{ds}\vert_{s=0} = P_{J,A_1,\ldots}(v).
\label{usefullittle}
\end{gather}
Substituting \ref{usefullittle} in \ref{infchange} we see that
\begin{gather}
0=d\langle \zeta, \mu_{w(0)} \rangle \left( P_{J,A_1,\ldots}(v)\right) = \Omega_{(J,A_1,A_2\ldots)}(P(\zeta),P(v)) \ \forall \ \zeta \in \mathfrak{k}^{\perp},
\end{gather}
thus implying that $P(v)=0$, i.e., $v\in \mathfrak{k}$ which is a contradiction. Hence $D_{\tilde{g}}G(0)$ is an isomorphism implying that $\Phi(x)=F_{\tilde{g}(x)}(\Phi_1(x))$ is the desired slice.
\end{proof}
Steps 3 and 4 are exactly the same as in \cite{Gabor, Mariotipler}. We now complete step 5 and hence the proof of theorem \ref{thm:perturb}.
\begin{proof} 
If $(M',J',(L_i'))$ is a small deformation of the complex structure of $(M,J,(L_i))$, then $c_1(L_i') = c_1(L_i)$. In particular, since $\omega\in c_1(L_0)$, we can use Moser's lemma to modify $L_0'$ and $J'$ by a small diffeomorphism so that $\omega_0$ is  Hermitian with respect to $J'$, and also tamed by it. That is, we can assume without loss of generality that $J'\in \sJ$. Next, since $L_i'$ is isomorphic to $L_i$ as a smooth line bundle, $h_i$ is also hermitian metric on $L_i'$. Let $A_i'$ be the corresponding Chern connection on $L_i'$. In this way, we obtain a point $(J',A') \in \sM$ in a small neighbourhood of $(J,A)$. By Theorem \ref{kuranishi}, possibly by modifying $(J',A')$ by the action of $\sG^\CC$, we can assume that there exists a $v'\in B$ such that $\Phi(v') = (J',A')$. \\
\indent If $v'$ is polystable for the $K^\CC$ action then by steps 2 and 3, we have a zero in the $\tilde \sG^\CC$ orbit of $(J',A')$, and hence coupled cscK metrics on $(M,J',(L_i'))$. If not, applying the Hilbert-Mumford criterion to $v'$, we may conclude that there exists a one-parameter subgroup $\rho : \mathbb{C}^{*} \rightarrow K^\CC$ such that $$v_0=\displaystyle \lim_{\lambda \rightarrow 0} \rho(\lambda) v'$$ satisfies $$\nu(v_0)=0.$$ Such a $v_0$ could potentially be located outside $B$. If so, using an element of $K^\CC$ we can bring it inside the ball.  Note that since integrability is a closed condition, $v_0$ also represents an integrable point. Since $\nu(v_0)=0$, by Step 4 above, we can perturb $v_0$ within its gauge orbit to $v_0'$ such that $\mu(\Phi(v_0')) = 0$. We let $\Phi(v_0'):=(J_0,A_0)$, and $(M,J_0,(L_{i,0}))$ denote the corresponding polarized tuple. There is of course a possibility that $v_0 = 0$, in which case we have $\Phi(v_0) = (J,A)$ and hence by the hypothesis, $\mu(\Phi(v_0)) = 0$, and we simply have $(J_0,A_0) = (J,A)$. In any case, since $\mu(J_0,A_0) = 0$, by Theorem  \ref{thm:momentmapthm}, $(M,J_0,(L_{i,0}))$ admits coupled cscK metrics. In particular, its Futaki invariant also vanishes. As in \cite{Gabor}, to complete the proof we  produce a  test configuration with $(M,J_0,(L_{i,0}))$ as the central fibre. \\
\indent The test configuration is constructed as follows. The Kuranishi map $\Phi$ from Proposition \ref{firststepslice}, along with $\rho$ produces an $S^1$-equivariant map from a small disc $F : \Delta \rightarrow \mathcal{N}$ such that $F(t) \simeq (J,A_1,\ldots) \ \forall \ t\neq 0$ and $F(0) \simeq (J',A_1'\ldots)$. Let $\pi:\tilde{V} = M \times \Delta\rightarrow \Delta$,  $\mathcal{L}_i = \pi_{1}^{*} L_i \ \forall \ i\geq 1$ with the almost complex structures given by $F(t)$ for every $t$, and the holomorphic structure on $\sL_i\Big|_{\pi^{-1}(t)}$ defined by the connections $A_i'$. These structures are integrable because $F$ is holomorphic. The $S^1$-action extends to a $\mathbb{C}^{*}$-action which then lifts in a natural manner to the line bundles (and hence to $\mathcal{L}=\otimes_{i=0}^{m} \mathcal{L}_i$). This family is flat over $\Delta$ because $\mathcal{V}$ and $\Delta$ are smooth and the dimension of the fibre is a constant. Hence this is a valid test configuration in the sense of section \ref{sec:Futaki} with central fibre $(M,J_0,(L_{i,0}))$.
The above discussion shows that the Futaki invariant of $(M,J_0,(L_{i,0}))$ vanishes, and hence by K-polystability, $(M,J',(L_{i}'))$ is isomorphic to the central fibre which admits coupled cscK metrics. Pulling these metrics back by the isomorphism, we get coupled cscK metrics on $(M,J',(L_i'))$. \\
%\indent Lastly, if $v'$ is strictly unstable, then there is a one-parameter subgroup $\rho$ whose limit $v_0=0$. By the above argument, we have a test configuration whose central fibre $(M,J_0,A_0)$ admits coupled cscK  metrics and hence has zero Futaki invariant. By K-polystability, the test configuration is trivial and hence $(M,J',A')$ is isomorphic to $(M,J_0,A_0)$ and we are done.
\end{proof}
\begin{appendix}
\section{Twisted Bergman kernel}

\setcounter{thm}{0}
     \renewcommand{\thethm}{\Alph{section}\arabic{thm}}

The aim of this appendix is to prove Corollary \ref{twisted expansion}. After the first draft of this paper appeared online, it was pointed to the authors by Ruadha\'i­ Dervan that the first parts of Theorem \ref{bergman} and Corollary \ref{twisted expansion}, have already been obtained by Keller in \cite{Keller} by using the usual method of ``peaked sections" of Tian \cite{Tian}.  The equivariant expansion of the Bergman kernel, as is well known, follows from a small modification of this proof, and can also probably be proved using a equivariant Riemann-Roch theorem applied to twisted line bundles (cf. \cite{Der}). For the convenience of the reader, we include an outline of the proof following  the exposition in \cite{Gaborbook}. More general results of this nature can be found in \cite{M-M}. Let $M^n$ be a K\"ahler manifold with two ample bundles $L_0$ and $L$ with Hermitian metrics $h_0 = e^{-\varphi_0}$ and $h = e^{-\varphi}$ with curvatures $F_{h_0}$, $F_h$ such that $\omega_0 = \sqrt{-1}F_{h_0}$ $\omega = \sqrt{-1}F_{h}$ are K\"ahler forms in $ 2\pi c_1(L_0)$ and $2\pi  c_1(L)$ respectively. We are interested in the space $H^0(M,L_0^k\otimes L^{-1})$ of holomorphic sections of $L^k_0\otimes L^{-1}$ with the $L^2$-inner product $$\langle s,t\rangle_{k} = \int_M\langle s,t\rangle_{h^k_0\otimes h}\dvok.$$ 

Let $w$ be a holomorphic vector field generating a $\CC^*$ action on $M$ and Hamiltonian with respect to both $\omega_0$ and $\omega$ with Hamiltonians $\theta_0$ and $\theta$. Our convention is that $\theta_0$ satisfies $$\sqrt{-1}\dbar \theta_0 = i_w\omega_0,$$ and similarly for $\theta$. The choice of the Hamiltonian $\theta_0$ is related to a (dual) action on $H^0(M,L^k_0)$ given by $$w \cdot s = \nabla^{(0)}_{-w} s - \sqrt{-1}\theta_0s,$$ where $\nabla^{(0)}$ is the Chern connection of $h_0$. Similarly $H$ induces an action on $H^0(M,L^{-1})$ and together they induce an action on $H^0(M,L_0^k\otimes L^{-1})$ which we denote by $ 2\pi \sqrt{-1}A_k$. Then $A_k$ is Hermitian. For any orthonormal basis $\{s_0,\cdots,s_{N_k}\}$, we define the {\em twisted Bergman kernel} by $$\rho_k(x) := \sum_{i=0}^{N_k}|s_i,s_i|^2_{h^k_0\otimes h}(x)$$ and the {\em equivariant twisted Bergman kernel} by $$\rho_k^{S^1}(x) := k^{-1}\sum_{i=0}^{N_k}\langle A_ks_i,s_i\rangle_{h^k_0\otimes h}(x).$$ In particular, if $\{s_{i}\}$ is an orthonormal basis of eigenvectors with eigenvalues $\{\la_i\}$, then $$\rho_k^{S^1}(x) := k^{-1}\sum_{i=0}^{N_k} \la_k|s_i|^2_{h^k_0\otimes h}(x).$$ 

\begin{thm}\label{bergman}
As $k\rightarrow \infty$ we have the following expansions 
\begin{align*}
(2\pi)^n\rho_k(x) &= 1 + \Big(\frac{S_{\omega_0}}{2} - \mathrm{tr}_{\omega_0}\omega\Big)k^{-1} + O(k^{-2})\\
(2\pi)^n\rho_k^{S^1}(x) &= -\theta_0 - \Big[\theta_0\Big(\frac{S_{\omega_0}}{2} - \mathrm{tr}_{\omega_0}\omega\Big)  -\theta \Big]k^{-1} + O(k^{-3/2}).
\end{align*}
\end{thm}

From general considerations the error in the second line should be $O(k^{-2})$, but our proof yields this slightly weaker result which is enough for our purposes. 
\begin{cor}\label{twisted expansion} With notation as above, we have 
\begin{align*}
(2\pi)^nd_{t,k} = a_{t,0}k^{n} + a_{t,1}k^{n-1} + O(k^{n-2})\\
(2\pi)^n\mathrm{Tr}(A_k) = b_{t,0}k^{n+1} + b_{t,1}k^{n} + O(k^{n-1}), 
\end{align*} where 
\begin{align*}
a_{t,0} &= \int_M\dvo,~ a_{t,1} = \int_M\Big(\frac{S_{\omega_0}}{2} - \mathrm{tr}_{\omega_0}\omega\Big)\dvo\\
b_{t,0} &= -\int_M\theta_0\dvo,~ b_{t,1} = -\frac{1}{2}\int_M\theta_0S_{\omega_0}\dvo + \int_M[\theta_0\mathrm{tr}_{\omega_0}\omega + \theta]\dvo.
\end{align*}
\end{cor}

\begin{proof}[Proof of Theorem \ref{bergman}] The required expansion is obtained using the  ``peak" sections method of Tian \cite{Tian}, and we first recall the relevant parts of this technique following the exposition in \cite{Gaborbook}. Fixing $x\in M$, the main idea is to construct a  holomorphic section $\eta$ of $L_0^k\otimes L^{-1}$ such that $\|\eta||_{L^2} = 1$ and $\eta$ is almost orthogonal (with an error of at most $O(k^{-2})$) to all holomorphic sections vanishing at $x$. It is easy to see that 
\begin{equation}\label{berg eq}
\rho_k(x) = \frac{|\eta(x)|^2_{h^k\otimes h}}{||\eta||^2_{k}},~\rho_k^{S^1}(x) = \frac{\langle k^{-1}A_k\eta,\eta\rangle_{h^k\otimes h}(x)}{||\eta||^2_{k}},
\end{equation} and so the theorem would follow from an expansion of $|\eta|^2(x)$ and $k^{-1}A_k$. Throughout we denote $\vep(k)$ to be any error term that is $O(k^{-N})$ for all $N$.

Suppose there exist normal coordinates $(w^1,\cdots,w^n)$ for $\omega_0$ on the unit ball $B = \{w\in \CC^n~|~|w|<1\}$ such that $\omega_0 = \ddbar\varphi_0$ where $$\varphi_0(w) = |w|^2 - \frac{1}{4}R_{i\bar j k \bar l}w^{i}\bar w^j w^k\bar w^l + Q_0(w) + P_0(w),$$ where $Q_0$ is a quintic polynomial, $|P_0(w)| = O(|w|^6)$ and $R_{i\bar j k\bar l}$ denotes the curvature of $ \omega_0$. Also we can choose the coordinates so as to diagonalize $\omega$, so that $2\pi\omega = \ddbar\varphi$, where $$\varphi(w) = \sum_i\la_i|w^i|^2 + Q(w) + P(w),$$ where $Q(w)$ is a cubic polynomial and $|P(w)| = O(|w|^4).$  Note that the factor of $2\pi$ in front means that $\Lambda_{\omega_0}\omega(x) = \sum_i\la_i$. It is convenient to rescale the coordinates $z^i = \sqrt{k}w^i$, so that ball now becomes $B = \{|z| \leq \sqrt{k}\}$ and $\Phi_0(z) = k\varphi_0(z)$ is then given by $$\Phi_0(z) = |z|^2 - \frac{k^{-1}}{4}R_{i\bar j k \bar l}z^{i}\bar z^j z^k\bar z^l + k^{-3/2}Q_0(z) + kP_0(k^{-1/2}z).$$ 
Note that $e^{-\Phi}$ is then the metric $h^k$. The aim is to take an ``almost holomorphic section" $\sigma_0$ such that $$|\sigma_0|_{h^k\otimes h^{-1}} = e^{-\Phi_0 + \varphi}$$ on $\{|z|< k^{-1/5}\}$, and $||\dbar \sigma_0||_{k} = \vep(k)$, and perturb it to a genuine holomorphic section $\sigma$ for $k>>1$. This relies on the invertibility of the Laplacian $\Delta_{\dbar} = \dbar^*\dbar + \dbar\dbar^*$ on $L^k$ valued $(0,1)$ forms, where the adjoint is computed using the $L^2$-inner product above. By the Weitzenbock formula $$ \Delta_{\dbar}= \overline{\nabla}^*\overline\nabla + \mathrm{Ric}_{k\omega_0} + \frac{\sqrt{-1}}{n}\Lambda_{k\omega_0}F_{h_0^k\otimes h^{-1}}.$$ Now for $k>>1$, we have $\mathrm{Ric}_{k\omega_0} \geq -\frac{1}{4}(k\omega_0)$. In the usual case, the curvature term $F$ is simply identity, and so the Laplacian is lower bounded by $1/2$, say. In our case,  since $\sqrt{-1}F_{h^k_0\otimes h^{-1}} = k\omega_0- \omega$, we have an extra term $\Lambda_{k\omega_0}\omega$ which goes to zero as $k\rightarrow \infty$, and so we also have $$\Delta_{\dbar} \geq \frac{1}{2}$$ for $k>>1$. As in the standard case we can then take $\sigma = \sigma_0 - \dbar^*\Delta_{\dbar}^{-1}\dbar \sigma_0$ to be the required holomorphic section. And just as in the standard case, we also have $|\sigma(x)|_{h^k_0\otimes h^{-1}} = 1 + \vep(k)$, $||\sigma - \sigma_0||_k = \vep(k),$ and that for every holomorphic section $\tau$  vanishing at $x$, $$|\langle \tau,\sigma\rangle_k| \leq Ck^{-1}||\tau||_k.$$ We now claim that 
\begin{equation}\label{L2 of section}
||\sigma||_k^2 = (2\pi)^n\Big[1 - \Big(\frac{S_{\omega_0}(x)}{2} - \Lambda_{\omega_0}\omega(x)\Big)k^{-1} + O(k^{-2})\Big].
\end{equation}
Up to an error of $\vep(k)$ it is enough to compute $L^2$ norm of $\sigma_0$ on $\{|z|<k^{1/5}\}$, which up to an $\vep(k)$ is the integral $$\int_{\CC^n}e^{-\Phi_0(z) + \varphi(z)}\frac{(\ddbar\Phi_0(z))^n}{n!}.$$ We have the expansions 
\begin{align}
e^{-\Phi_0(z) + \varphi(z)} = e^{-\sum_i\Big(1-\frac{\la_i}{k}\Big)|z^i|^2}\Big(1+ Q(k^{-1/2}z) + P(k^{-1/2}z) + \frac{k^{-1}}{4}R_{i\bar j p\bar q}z^iz^{\bar j}z^{p}z^{\bar q} \\- k^{-3/2}Q_0(z) - kP_0(k^{-1/2}z)\Big),\nonumber\\
\frac{(\ddbar\Phi_0(z))^n}{n!} = \Big(1 - k^{-1}R_{p\bar q}z^pz^{\bar q} + k^{-3/2}q_0(z) + O(k^{-2}|z|^4)\Big)\,dV,
\end{align}
where $dV$ is $2^n$ times the Euclidean volume, that is,$$dV = (\sqrt{-1})^ndz^1\wedge d\bar z^1 \wedge\cdots\wedge dz^n\wedge d\bar z^n.$$ 
The leading order term in the expansion is given by 
\begin{align*}
\int_{\CC^n}e^{-\sum_i\Big(1-\frac{\la_i}{k}\Big)|z^i|^2}\,dV &= \frac{(2\pi)^n}{\Pi_{i}\Big(1-\frac{\la_i}{k}\Big)} \\
&=(2\pi)^n\Big[1 + (\Lambda_{\omega_0}\omega) k^{-1} + O(k^{-2})\Big].
\end{align*}
As in the standard case, there is a contribution of order $O(k^{-1})$ by the terms $$\int_{\CC^n}e^{-|z|^2}\Big(\frac{1}{4}R_{i\bar j p\bar q}z^iz^{\bar j}z^{p}z^{\bar q} -R_{p\bar q}z^pz^{\bar q} \Big)\,dV = -(2\pi)^n\frac{S_{\omega_0}(x)}{2} .$$ The only ``new" terms we need to worry about are the ones involving $Q(k^{-1/2}z)$ and $P(k^{-1/2}z)$. Now $Q$ is a cubic polynomial, and hence by symmetry $$\int_{\CC^n} e^{-\sum_i\Big(1-\frac{\la_i}{k}\Big)|z^i|^2}Q(k^{-1/2}z)\,dV = 0.$$ On the other hand, the leading order term in $P(k^{-1/2}z)$ is $k^{-2}p_4(z)$, where $p_4$ is a fourth degree polynomial and hence it follows that $$\int_{\CC^n} e^{-\sum_i\Big(1-\frac{\la_i}{k}\Big)|z^i|^2}P(k^{-1/2}z)\,dV  = O(k^{-2}).$$ This completes the proof of \eqref{L2 of section}. As in the standard case, it follows that if $E\subset H^0(M,L_0^k\otimes L^{-1})$ is the co-dimension one subspace of sections vanishing at $x$, and $\eta$ is the projection of $\sigma$ to the orthogonal complement, then $|\eta(x)|^2_{h^k_0\otimes h^{-1}} = 1+\vep(k)$ and 
\begin{equation}\label{L2 of final section}
||\eta||_k^2 = (2\pi)^n\Big[1 - \Big(\frac{S_{\omega_0}(x)}{2} - \Lambda_{\omega_0}\omega(x)\Big)k^{-1} + O(k^{-2})\Big].
\end{equation}
Combined with \eqref{berg eq}, this gives the expansion of $\rho_k$. It now remains to analyze $k^{-1}A_k$. As in \cite[pg.\ 139-140]{Gaborbook}, if we denote the Chern connection on $L_0^k\otimes L$ by $\tilde \nabla $, then $$(k^{-1}\tilde \nabla_w)\eta(x) = O(k^{-3/2}).$$ The action of $k^{-1}A_k$ is given by 
\begin{align*}
k^{-1}A_k\cdot \eta &= \frac{1}{k\sqrt{-1}}\tilde \nabla_{-w} - \theta_0\eta + k^{-1}\theta\eta \\
&=- \theta_0\eta + k^{-1}\theta\eta + O(k^{-3/2}).
\end{align*} 
Combining this with \eqref{berg eq} and \eqref{L2 of final section}, and recalling that $|\eta(x)|_{h_0^k\otimes h^{-1}}^2 = 1+\vep(k)$, we obtain
\begin{align*}
(2\pi)^n\rho_k^{S^1}(x) &= \frac{(-\theta_0 + k^{-1}\theta + O(k^{-3/2}))(1+\vep(k))}{1 -  \Big(\frac{S_{\omega_0}(x)}{2} - \Lambda_{\omega_0}\omega(x)\Big)k^{-1} + O(k^{-2})}\\
&=\Big(-\theta_0 + k^{-1}\theta\Big) \Big(1 +  \Big(\frac{S_{\omega_0}(x)}{2} - \Lambda_{\omega_0}\omega(x)\Big)k^{-1} + O(k^{-2})\Big) + O(k^{-3/2})\\
&=-\theta_0 - \Big(\theta_0\Big(\frac{S_{\omega_0}(x)}{2} - \Lambda_{\omega_0}\omega(x)\Big) - \theta\Big)k^{-1} + O(k^{-3/2}).
\end{align*}
\end{proof}

\end{appendix}

%\begin{prop}\label{Donaldson formula} $$\ctFut(W,w) = 4\pi\frac{a_1b_0 - a_0b_1}{a_0^2} + \frac{b_0}{a_0}-\sum_{i=1}^m\frac{b_{0,j}}{a_{0,j}} + (1-t)\sum_{i=1}^m\Big(\frac{\tilde b_{0,j}a_{0,j} - b_{0,j}\tilde a_{0,j}}{a_{0,j}^2}  \Big).$$\end{prop}
%%%%%%%%%%%%%%%%%%%%%%%%%%%%%
 %%%%%%%%%%%%%%%%%%%%%%%%%%%%%%%%%%%%%%%%%%%%%%

\end{document}